\documentclass[11pt,reqno]{amsart}
\usepackage{style}
\usepackage[foot]{amsaddr}
\usepackage[margin=1in]{geometry}

\title{Spectral Recovery in the Labeled SBM}
\author{Julia Gaudio{$^\star$}, Heming Liu{$^\dagger$}}
\address[{$^\star$}]{Department of Industrial Engineering, Northwestern University}

\begin{document}

\begin{abstract}
We consider the problem of exact community recovery in the Labeled Stochastic Block Model (LSBM) with $k$ communities, where each pair of vertices is associated with a label from the set $\{0,1, \dots, L\}$. A pair of vertices from communities $i,j$ is given label $\ell$ with probability $p_{ij}^{(\ell)}$, and the goal is to recover the community partition. We propose a simple spectral algorithm for exact community recovery, and show that it achieves the information-theoretic threshold in the logarithmic-degree regime, under the assumption that the eigenvalues of certain parameter matrices are distinct and nonzero. Our results generalize recent work of Dhara, Gaudio, Mossel, and Sandon (2023), who showed that a spectral algorithm achieves the information-theoretic threshold in the Censored SBM, which is equivalent to the LSBM with $L = 2$. Interestingly, their algorithm uses eigenvectors from two matrix representations of the graph, while our algorithm uses eigenvectors from $L$ matrices.
\end{abstract}

\maketitle

\section{Introduction}
\label{sec:introduction}

Community detection is an important inference task, with applications to biological, physical, and social networks. The goal in community detection is to recover latent community structure in a graph. The Stochastic Block Model (SBM), introduced by Holland, Laskey, and Leinhardt \cite{Holland1983}, is considered the canonical probabilistic generative model for graphs with community structure. In the SBM, the vertices of a graph are randomly assigned a community label from the set $\{1, 2, \dots, k\}$, where $\pi_i$ denotes the probability that  given vertex has label $i$. Given these community assignments, a pair of vertices from communities $(i,j)$ are connected by an edge with probability $p_{ij}$, independently. The SBM has attracted significant attention in the probability, statistics, machine learning, and information theory literature; see the survey of Abbe \cite{Abbe2017}.

In this paper, we consider a generalization of the SBM, known as the Labeled Stochastic Block Model (LSBM), originally proposed by Heimlicher, Lelarge, and Massouli\'e \cite{Heimlicher2012}. In the LSBM, each pair of vertices $(u,v)$ is associated with a label $\mathcal{L}(u,v) \in \{0,1, \dots, L\}$. Each pair of communities $(i,j) \in [k]^2$ is associated with a label distribution, where $p_{ij}^{(\ell)}$ is the probability that a pair of vertices $(u,v)$ is given label $\mathcal{L}(u,v) = \ell$, conditioned on $u$ belonging to community $i$ and $v$ belonging to community $j$.
The goal is to recover the community partition, given the labels $\{\mathcal{L}(u,v)\}_{u,v \in V}$.

Our focus is on the exact recovery problem, and we consider a particular regime, roughly corresponding to the well-studied logarithmic degree regime in the SBM. We propose a simple spectral algorithm for recovering the community labels, and show that it achieves the information-theoretic threshold for almost all parameters. 

\paragraph*{Related Work}
The LSBM was introduced by Heimlicher, Lelarge, and Massouli\'e \cite{Heimlicher2012}, who investigated the sparse version of the problem (and therefore, the problem of partially recovering the communities). The sparse regime was further studied by Lelarge, Massouli\'e, and Xu \cite{Lelarge2015}, we found recovery guarantees for min-bisection, semidefinite programming, and a spectral algorithm. Subsequently, Yun and Proutiere \cite{Yun2016} studied partial, almost exact, and exact recovery in the LSBM. In particular, Yun and Proutiere \cite{Yun2016} determined the information-theoretic threshold for exact recovery, and provided an efficient algorithm achieving the information-theoretic limit. Our main contribution is to provide a simpler algorithm which also achieves the information-theoretic threshold (under a certain technical condition).

The LSBM is related to the Censored Stochastic Block Model (CSBM), which corresponds to the LSBM with $L= 2$. In the CSBM, we observe a censored graph, where we can only observe the connectivity status (present or absent edge) between some pairs of vertices, thus leading to three possible observations between pairs of vertices (present, absent, or censored). The CSBM has been studied in several works \cite{Abbe2014,Hajek2016,Dhara2022a,Dhara2022c}. The most relevant work to ours is the paper of Dhara, Gaudio, Mossel, and Sandon \cite{Dhara2022c}, who proposed a spectral algorithm for exact recovery in the CSBM. Their algorithm encodes the graph into two \emph{signed} adjacency matrices, each using a different ternary numerical encoding of the three possible observations. It turns out that using two matrix encodings is essential in achieving the information-theoretic threshold; for nearly all model parameters, any spectral algorithm based on a single matrix will not achieve the information-theoretic threshold \cite{Dhara2022c}. Our work shows that an appropriate generalization of their algorithm with $L$ matrices also succeeds down to the information-theoretic threshold, confirming the prediction of Dhara et. al. \cite{Dhara2022c} (under a distinctness of eigenvalues assumption).

Our work adds to the growing body of examples of spectral algorithms for community detection. Spectral algorithms for the SBM have a decades-long history, initiated in the work of McSherry \cite{McSherry2001}. However, it was only recently that Abbe, Fan, Wang, and Zhong \cite{Abbe2019} showed that a spectral algorithm achieves the information-theoretic threshold for exact recovery in the two-community SBM, without the need for a cleanup phase. The algorithm simply computes the second leading eigenvector of the adjacency matrix of the graph, and thresholds the entries at $0$ to assign communities. Following the work of Abbe et. al. \cite{Abbe2019}, several works proposed cleanup-free spectral algorithms for other inference problems: the planted dense subgraph model \cite{Dhara2022b}, submatrix localization \cite{Dhara2022b}, and community detection in the presence of side information, in both the SBM and Gaussian models \cite{Gaudio2024}. Spectral algorithms have been derived not only for adjacency matrices, but also for Laplacians \cite{Deng2021} and the so-called ``similarity matrix'' of a hypergraph \cite{Gaudio2023}.
The common theme in the analysis of these algorithms is \emph{entrywise eigenvector analysis}, developed by Abbe et. al. \cite{Abbe2019}; we crucially use these results in our analysis.

\noindent \textbf{Notation} 
The infinity norm of a vector $x$ is denoted $\Vert x \Vert_{\infty} := \max_i |x_i|$. Given two matrices $A,B$ with the same dimensions, the notation $\langle A , B \rangle := \text{tr}(A B^T)$ denotes the standard matrix inner product. We use Bachmann--Landau asymptotic notation in terms of the number of vertices, $n$.

\noindent \textbf{Organization} The rest of this paper is structured as follows. Section \ref{sec:model} formally introduces the LSBM along with the information-theoretic threshold for exact recovery. Section \ref{sec:algorithm} describes the spectral algorithm, and Section \ref{sec:proof} contains the proof of exact recovery.

\section{Model and Information-Theoretic Threshold}\label{sec:model}
\subsection{Model Definition}
We now describe the LSBM in the logarithmic degree regime. 
\begin{definition}
The LSBM is described by parameters $k, L \in \mathbb{N}$, $\pi \in \mathbb{R}^k$, $q_{ij}^{(\ell)}$ for $i,j \in [k], \ell \in [L]$, and $t > 0$. A graph on $n$ vertices is sampled as follows:
\begin{enumerate}
    \item Each vertex is independently assigned a community label, where label $i$ is assigned with probability $\pi_i$, for $i \in [k]$.
    \item Conditioned on the community assignments, each pair of vertices is given a label from the set $\{0,1, \dots, L\}$. A pair of vertices $(u,v)$ in communities $(i,j)$, is given label $\mathcal{L}(u,v) \in \{0, 1, \dots, L\}$. The label is $\ell$ with probability $\frac{t \log n}{n} \cdot q_{ij}^{(\ell)}$, for $\ell \in [L]$. With probability $1- \frac{t \log n}{n}$, the label is $0$.
\end{enumerate}
We denote a graph $G$ generated with these parameters by  $G \sim \text{LSBM}(\pi,q,t,n)$.
\end{definition}
Note that for every pair $i,j$, the values $\{q_{ij}^{(\ell)}\}_{\ell \in [L]}$ must correspond to a valid probability distribution; that is, $\sum_{\ell = 1}^{L} q_{ij}^{(\ell)} = 1$. With this parametrization, the $0$ label is most frequent. Since the $0$ label appears with the same probability for every pair $(i,j)$, it is also uninformative. Therefore, another interpretation of this model is that we observe useful information (i.e., a label from $\{1, 2, \dots, L\}$) for a given pair with probability $\frac{t \log n}{n}$, and the label is $\ell$ with probability $q_{ij}^{(\ell)}$ for a pair of vertices in communities $i,j$. In the case where $L=2$, this parameterization captures the version of the CSBM which was studied by \cite{Hajek2016,Dhara2022a,Dhara2022c}.

\subsection{Information-Theoretic Threshold}
The information-theoretic threshold is stated in terms of a Chernoff--Helllinger divergence \cite{Abbe2015}.
\begin{definition}
Let $x, y \in \mathbb{R}^m$ for some $m \in \mathbb{N}$. The Chernoff--Hellinger divergence between $x$ and $y$, denoted $D_+(x,y)$, is given by
\[D_+(x,y) = \sup_{\lambda \in [0,1]} \left\{\sum_{i=1}^m \lambda x_i + (1-\lambda) y_i - x_i^{\lambda} y_i^{1-\lambda} \right\}.\]
The definition straightforwardly extends to multidimensional (e.g., matrix-valued) $x$ and $y$. Also note that if $x$ and $y$ represent probability distributions, then
\begin{equation}
D_+(x,y) = 1 - \inf_{\lambda \in [0,1]} \left\{\sum_{i=1}^m x_i^{\lambda} y_i^{1-\lambda} \right\}.  \label{eq:D_+-prob}  
\end{equation}
\end{definition}
We say that an estimator achieves exact recovery if it determines the correct community partition with probability $1-o(1)$. The following result characterizes the information-theoretic threshold in the regime described above, and follows directly from \cite[Theorem 3]{Yun2016} (see the Appendix for details). See also \cite[Theorem 1]{Dreveton2023} which determines the IT threshold for a much more general setting.
\begin{theorem}\label{theorem:IT}
Consider $\pi \in \mathbb{R}^k$, $\{q_{ij}^{(\ell)}\}_{i,j \in [k], \ell \in [L]}$, and $t > 0$. Let $G_n \sim \text{LSBM}(\pi, q, t, n)$. For $i \in [k]$, let $\theta_i$ be the $k \times L$ matrix whose $(j,\ell)$ entry is $\pi_j q_{ij}^{(\ell)}$. Define 
\begin{equation}t_c = \left(\min_{i \neq j} D_+(\theta_i, \theta_j) \right)^{-1}. \label{eq:critical-t}
\end{equation}
\begin{enumerate}
    \item If $t > t_c$, then the MAP estimator achieves exact recovery. 
    \item If $t < t_c$, then any estimator fails to achieve exact recovery. Moreover, any estimator fails to recover the community partition with probability $1-o(1)$.
\end{enumerate}
\end{theorem}

\section{Spectral Algorithm}\label{sec:algorithm}
We propose a spectral algorithm for community recovery, which is an extension of the $L=2$ case covered by \cite{Dhara2022c}. At a high level, we construct $L$ matrix representations of the input graph, and compute multiple carefully weighted linear combinations of the leading eigenvectors of these matrices in order to infer the community partition. The design of our algorithm hinges on a precise entrywise control over the leading eigenvectors, using the machinery developed  by Abbe, Fan, Wang, and Zhong \cite{Abbe2019}.

In more detail, the $\ell^{\text{th}}$ matrix has the following form:
\begin{align}
A^{(\ell)}_{uv} =
\begin{cases}
0 & \mathcal{L}(u,v) \neq \ell \\
1 & \mathcal{L}(u,v) =\ell.
\end{cases} \label{eq:signed-adjacency-matrix}
\end{align}
Let $u_1^{(\ell)}, \dots,  u_k^{(\ell)}$ be the top $k$ leading eigenvectors of $A^{(\ell)}$. Informally, the results of Abbe et. al. \cite{Abbe2019} show that 
\[u^{(\ell)}_i \approx \frac{A^{(\ell)} u_i^{\star (\ell)}}{\lambda_i^{\star (\ell)}}\]
for $i \in [k]$, where $(\lambda_i^{\star (\ell)},u_i^{\star (\ell)})$ is the $i^{\text{th}}$ leading eigenpair of $\mathbb{E}[A^{(\ell)} \mid \sigma^{\star}]$, and $\approx$ indicates proximity in $\ell_{\infty}$. (For now we are disregarding the fact that eigenvectors are determined up to sign.) But then also 
\begin{equation}
\sum_{i=1}^k \sum_{\ell = 1}^L c_i^{(\ell)}u_i^{(\ell)} \approx \sum_{i=1}^k \sum_{\ell = 1}^L c_i^{(\ell)}\frac{A^{(\ell)} u_i^{\star (\ell)}}{\lambda_i^{\star (\ell)}} \label{eq:spectral-intuition}    
\end{equation}
for any constants $\{c_i^{(\ell)}: i \in [k], \ell \in [L]\}$. 

Observe that since $\mathbb{E}[A^{(\ell)} \mid \sigma^{\star}]$ is a block matrix, the vectors $u_i^{\star (\ell)}$ themselves have a block structure organized according to the community partition (that is, for a given community $m \in [k]$, we have that $u_{iw}^{\star (\ell)}$ takes a constant value for all $w$ such that $\sigma^{\star}(w) = m$). Now consider the $v^{\text{th}}$ entry of the vector on the right hand side of \eqref{eq:spectral-intuition}, corresponding to vertex $v$: \[\sum_{w \neq v} \sum_{i \in [k]} \sum_{\ell \in [L]} \frac{c_i^{(\ell)}}{\lambda_i^{\star (\ell)}}  A^{(\ell)}_{vw} u_{iw}^{\star (\ell)},\]
which we can write as
\small
\begin{align*}
\sum_{w \neq v} \sum_{j \in [k]} \sum_{\ell' \in [L]} \sum_{i \in [k]} \sum_{\ell \in [L]} \frac{c_i^{(\ell)}}{\lambda_i^{\star (\ell)}}  A^{(\ell)}_{vw} u_{iw}^{\star (\ell)}  \mathbbm{1}\left\{\sigma^{\star}(w) = j, \mathcal{L}(v,w) = \ell'  \right\}.    
\end{align*}
\normalsize
Consider some $j \in [k]$ and $\ell' \in [L]$. Due to the block structure of the vectors $u_i^{\star (\ell)}$, we see that the inner summation $\sum_{i \in [k]} \sum_{\ell \in [L]} \frac{c_i^{(\ell)}}{\lambda_i^{\star (\ell)}}  A^{(\ell)}_{vw} u_{iw}^{\star (\ell)} $ is constant for all $w$ with $\sigma^{\star}(w) = j, \mathcal{L}(v,w) = \ell'$. Therefore, the $v^{\text{th}}$ entry of \eqref{eq:spectral-intuition} is of the form 
\begin{align}
\sum_{w \neq v} \sum_{j \in [k]} \sum_{\ell' \in [L]} C_{j,\ell'} \mathbbm{1}\left\{\sigma^{\star}(w) = j, \mathcal{L}(v,w) = \ell'  \right\}&= \sum_{j \in [k]} \sum_{\ell \in [L]}  C_{j,\ell} \cdot |\{w : \sigma^{\star}(w) = j, \mathcal{L}(v,w) = \ell \}|. \label{eq:spectral-form} 
\end{align}
That is, the $v^{\text{th}}$ entry is a weighted sum of label counts incident to $v$. We will see that certain choices of coefficients $\{C_{j,\ell}: j \in [k], \ell \in [L]\}$ are useful for clustering. In turn, we can design the weights $\{c_i^{(\ell)} : i \in [k], \ell \in [L]\}$ to achieve the desired coefficients. 

To determine the coefficients, we turn to the so-called \emph{genie estimator}. Fixing a vertex $v$, the genie estimator is tasked with determining the label of $v$, knowing $\sigma^{\star}(w)$ and $\mathcal{L}(v,w)$ for all $w \neq v$. The genie estimator thus computes the MAP estimator for $v$; that is,
\begin{align*}
&\argmax_{i \in [k]}\left\{ \pi(i) \prod_{j \in [k]} \prod_{\ell \in [L]} \left(q_{i,j}^{(\ell)}\right)^{|\{w \neq v: \sigma^{\star}(w) = j, \mathcal{L}(v,w) = \ell\}|} \right\}\\
&=\argmax_{i \in [k]} \Big\{\log\left(\pi(i) \right)  +  \sum_{j \in [k]} \sum_{\ell \in [L]}\log\left(q_{ij}^{(\ell)} \right)  |\{w \neq v: \sigma^{\star}(w) = j, \mathcal{L}(v,w) = \ell\}| \Big\}.
\end{align*}
Observe that the $\log(\pi(i))$ term is insignificant, meaning that the MAP is well-approximated by the MLE. 
Comparing the above to \eqref{eq:spectral-form}, we see that the spectral algorithm is able to mimic the form of the genie estimator. Essentially, the values $\{\log (q_{ij}^{(\ell)}) : j \in [k], \ell \in [L]\}$ take the place of $\{C_{j,\ell} : j \in [k], \ell \in [L]\}$ in \eqref{eq:spectral-form}, for each $i \in [k]$ in turn. The algorithm is as follows.
\begin{breakablealgorithm}
\caption{Spectral Algorithm}\label{alg:spectral}
\begin{algorithmic}[1]
\Require{Parameters $\pi \in (0,1)^k$ and $\{q_{ij}^{(\ell)} : i,j \in [k], \ell \in [L]\}$; graph $G$ on $n$ vertices}
\vspace{.2cm}
\Ensure{Community classification $\hsig\in [k]^n$.}
\vspace{0.2cm}
    \State For each $\ell \in [L]$, let $A^{(\ell)}$ be constructed from $G$ according to \eqref{eq:signed-adjacency-matrix}.
    \State For each $\ell \in [L]$, find the top $k$ eigenpairs of $A^{(\ell)}$, respectively denoting them $(\lambda_1^{(\ell)}, u_1^{(l)}), \dots, (\lambda_k^{(l)}, u_k^{(l)})$. Let $U^{(l)}$ be the $n \times k$ matrix whose $i$-th column is $u_i^{(l)}$. 
    \State\label{step:weights}Use Algorithm \ref{alg:find-weights} to compute the weights  $\{c_{ij}^{(\ell)} : i,j \in [k], \ell \in [L]\} \subset \mathbb{R}$. Let $c_i^{(\ell)} \in \mathbb{R}^k$ be the vector whose $j^{\text{th}}$ entry is $c_{ij}^{(\ell)}$.
    \State \label{step:argmax} For $s^{(l)} \in \{\pm 1\}^{k}$, let $D^{\sss (s^{(l)})}:= \diag(s^{(l)})$. Construct the estimator 
    \begin{equation}\label{estimator-k-com}
    \hat{\sigma}(v;s^{(1)}, \dots, s^{(L)}) = \argmax_{i \in [k]} \left\{ \sum_{\ell=1}^L\left(U^{(\ell)}  D^{\sss (s^{(\ell)})}  c_i^{(\ell)}\right)_v \right\}.
    \end{equation}
    \State Return $\hsig$ which maximizes the posterior probability $\mathbb{P}(\hat{\sigma} \mid G )$ over all $\sigma \in \{\sigma(\cdot; s^{(1)}, \dots, s^{(L)}) : s^{(1)}, \dots s^{(L)} \in \{\pm 1\}^k\}$\label{step:find-signs}
\end{algorithmic}
\end{breakablealgorithm}
Since eigenvectors are determined up to sign, we need to allow for all possible signs; this is accomplished by introducing a set of diagonal matrices in Line \ref{step:argmax}, and dissambiguating the signs in Line \ref{step:find-signs} by taking the estimator with the highest posterior probability. The algorithm for finding the weights is given below.

\begin{breakablealgorithm}
\caption{Find Weights}\label{alg:find-weights}
\begin{algorithmic}[1]
\Require{Parameters $\pi \in (0,1)^k$ and $\{q_{ij}^{(\ell)} : i,j \in [k], \ell \in [L]\}$}
\vspace{.2cm}
\Ensure{Weights $\{c_i^{(\ell)} : i \in [k], \ell \in [L]\}$}
\vspace{0.2cm}
\State For $m \in [k]$, let $\cV_m:= \{i: n \sum_{j=0}^{m-1} \pi_j \leq i\leq  n \sum_{j=1}^{m} \pi_j \}$ with $\pi_0 =0$.
 \State For $\ell \in [L]$, let $B^{(\ell)} \in \mathbb{R}^{n \times n}$ be a symmetric block matrix, where 
 \[B^{(\ell)}_{uv} = \frac{t \log n}{n}q_{ij}^{(\ell)}\]
 for $u \in \cV_i, v \in \cV_j$.
\State Compute the top $k$ eigenpairs of $\{B^{(\ell)}\}_{\ell = 1}^L$, respectively denoting them by $\{(\gamma_1^{(\ell)}, v_1^{(\ell)}),\dots, (\gamma_k^{(\ell)}, v_k^{(l)})\}_{\ell=1}^L$. 
\State For $i \in [k], \ell \in [L]$, let $z_{i}^{(\ell)}$ be a block vector with $z_{i,v}^{(\ell)} = \log \left(q_{ij}^{(\ell)} \right)$ for $v \in \cV_j$.
 
\State\label{step:z} Return $\{c_{ij}^{(\ell)} : i,j \in [k], \ell \in [L]\} \subset \mathbb{R}$ satisfying 
\begin{align*}
\sqrt{n} \log(n) \left(\sum_{j=1}^k  c_{ij}^{(\ell)} \frac{v_j^{(\ell)}}{\gamma_j^{(\ell)}} \right) &= z_i^{(\ell)}  & \forall i \in [k], \ell \in [L].
\end{align*}
\end{algorithmic}
\end{breakablealgorithm}
In Algorithm \ref{alg:find-weights}, the matrix $B^{(\ell)}$ acts as a substitute for $A^{(\ell) \star} = \mathbb{E}[A^{(\ell)} \mid \sigma^{\star}]$, which is unknown. In turn, the eigenpairs of the matrices $\{B^{(\ell)}\}_{\ell = 1}^L$ allow us to solve for the weights $\{c_{ij}^{(\ell)}\}_{i,j \in [k], \ell \in [L]}$ that yield the vectors $\{z_i^{(\ell)}\}_{i\in [k], \ell \in [L]}$ in Step \ref{step:z}. The values of the $z_{i}^{(\ell)}$ vectors are chosen so that for each $i$ and $v \in V$, the sum $\sum_{\ell = 1}^L A_{v\cdot }^{(\ell)} \cdot z_{i}^{(\ell)}$ is proportional to $\langle W(i), d(v) \rangle$. 

The following result gives a guarantee for the success of Algorithm \ref{alg:spectral}. Let $Q^{(\ell)} \in \mathbb{R}^{k \times k}$ be the matrix where $Q_{ij}^{(\ell)} = q_{ij}^{(\ell)}$.
\begin{theorem}\label{theorem:spectral}
Consider $\pi \in \mathbb{R}^k$, $\{q_{ij}^{(\ell)}\}_{i,j \in [k], \ell \in [L]}$, and $t > t_c$ where $t_c$ is defined in \eqref{eq:critical-t}.  Suppose that $Q^{(\ell)} \cdot \diag(\pi)$ has $k$ distinct, nonzero eigenvectors, for each $\ell \in [L]$. Let $G \sim \text{LSBM}(\pi, q, t, n)$. Then with probability $1-o(1)$, Algorithm \ref{alg:spectral} with input $(\pi, q, G)$ exactly recovers the community partition.
\end{theorem}

\section{Proof of Exact Recovery}\label{sec:proof}

The proof of Theorem \ref{theorem:spectral} relies on two key results. The first is a characterization of the leading eigenvectors of adjacency matrices. Throughout, we fix an assignment $\sigma^{\star}$ such that the community sizes are close to their expectations. More precisely, letting $n_j = n_j(\sigma^{\star}) = |\{v : \sigma^{\star}(v) = j\}|$, we assume
\begin{equation}
|n_j - n\pi_j| \leq n^{2/3}  \label{eq:balanced-sigma}  
\end{equation}
(which holds with high probability, as each $n_j$ is marginally distributed as $\text{Bin}(n,\pi_j)$).
\begin{lemma}\label{lemma:entrywise}
Fix $\sigma^{\star}$ satisfying \eqref{eq:balanced-sigma}, and let $\ell \in [L]$. Let $A = A^{(\ell)}$, recalling  \eqref{eq:signed-adjacency-matrix}. Let $A^{\star} = \mathbb{E}[A\mid \sigma^{\star}]$. Let $(\lambda_i, u_i)_{i=1}^k$ and $(\lambda_i^{\star}, u_i^{\star})_{i=1}^k$ be the $k$ leading eigenpairs of $A$ and $A^{\star}$, respectively. Suppose that $Q^{(\ell)} \cdot \diag(\pi)$ has $k$ distinct, nonzero eigenvalues. Then for all $i \in [k]$,
\[\min_{s \in \{\pm 1\}} \left \Vert s u_i - \frac{A u_i^{\star}}{\lambda_i^{\star}} \right \Vert_{\infty} \leq \frac{C}{\log \log(n) \sqrt{n}}\]
with probability $1-O(n^{-3})$.
Here $C > 0$ is a constant depending on $\{q_{ij}^{(\ell)}\}_{i,j \in [k]}$ and $\{\pi_i\}_{i \in [k]}$.
\end{lemma}
The proof follows directly from \cite[Corollary 35]{Dhara2022c} together with \cite[Remark 44]{Dhara2022c}, and relies on the entrywise eigenvector analysis of \cite{Abbe2019}

The second key result (which is similar to \cite[Corollary 19]{Dhara2022c}) ensures the separation of \emph{degree profiles} arising from different communities. For a vertex $v$, let $d(v) \in \mathbb{R}^{k \times L}$ be the \emph{degree profile} of $v \in V$, where the $(j,\ell)$ entry of $d(v)$ is equal to $|\{u : \sigma^{\star}(u) = j, \mathcal{L}(u,v) = \ell\}|$. 
\begin{lemma}\label{lemma:concentration}
For each $i \in [k]$, let $W(i) \in \mathbb{R}^{k \times L}$ be the matrix whose $(j,\ell)$ entry is equal to $\log(q_{ij}^{(\ell)})$. Suppose $t > t_c$. Then there exists $\delta > 0$ such that with probability $1-o(1)$, 
\[\langle W(i) , d(v) \rangle \geq \max_{j \neq i} \langle W(j) , d(v)\rangle + \delta \log n\]
for all $i \in [k]$ and $v$ such that $\sigma^{\star}(v) = i$.
\end{lemma}
\begin{proof}
Recall that we condition on a fixed $\sigma^{\star}$ satisfying \eqref{eq:balanced-sigma}. It suffices to show that there exist $\delta, \eta > 0$ satisfying the following: for all $i,j \in [k]$ where $i \neq j$ and $v$ such that $\sigma^{\star}(v) = i$, we have
\begin{equation}
\langle (W(i) - W(j)) , d(v) \rangle > \delta \log n \label{eq:degree-profile}
\end{equation}
with probability $1-n^{-1-\eta}$.

We employ a Chernoff bound strategy to prove the desired claim. Fix distinct $i, j \in [k]$ and $v$ such that $\sigma^{\star}(v) = i$. For any $\lambda > 0$, the Markov inequality implies
\begin{align}
\mathbb{P}\left(\langle W(i) - W(j),  d(v)\rangle \leq \delta \log n \right) &= \mathbb{P}\left(e^{-\lambda \left(\langle W(i) - W(j), d(v)\rangle\right)} \geq e^{-\lambda \delta \log n} \right) \nonumber\\
&\leq n^{\lambda \delta} \mathbb{E}\left[e^{-\lambda \left(\langle W(i) - W(j),  d(v) \rangle\right)} \right]. \label{eq:Chernoff}
\end{align}
Next, observe that
\begin{align*}
\langle W(i) - W(j), d(v)\rangle &= \sum_{m \in [k]} \sum_{u \neq v : \sigma^{\star}(u) = m}\sum_{\ell \in [L]} \log\left(\frac{q_{im}^{(\ell)}}{q_{jm}^{(\ell)}}\right) \mathbbm{1}\{\mathcal{L}(v,u) = \ell\}.   
\end{align*}
Note that the innermost summation can be represented by a random variable $X_u$ which takes value $\log\left(\frac{q_{im}^{(\ell)}}{q_{jm}^{(\ell)}}\right)$ with probability $\frac{t \log n}{n} \cdot q_{im}^{(\ell)}$, for all $\ell \in [L]$, and with probability $1 - \frac{t \log n}{n}$ takes value $0$ (where $\sigma^{\star}(u) = m$). Due to independence of $\{\mathcal{L}(v,u) \}_{u \neq v}$ conditioned on the community assignments, we obtain
\begin{align*}
\mathbb{E}\left[e^{-\lambda \left(\langle W(i) - W(j) , d(v) \rangle\right)} \right] &= \prod_{m \in [k]} \left(1 - \frac{t \log n}{n} + \sum_{\ell \in [L]} \frac{t \log n}{n} \cdot q_{im}^{(\ell)} e^{-\lambda \log\left(\frac{q_{im}^{(\ell)}}{q_{jm}^{(\ell)}}\right)} 
  \right)^{n_m}\\
&=  \prod_{m \in [k]} \left(1 - \frac{t \log n}{n} + \frac{t \log n}{n}\sum_{\ell \in [L]} \left(q_{im}^{(\ell)} \right)^{1-\lambda} \left(q_{jm}^{(\ell)} \right)^{\lambda}\right)^{n_m}\\
&\leq \exp\left(-\frac{t \log n}{n}\sum_{m \in [k]} n_m \left(1- \sum_{\ell \in [L]} \left(q_{im}^{(\ell)} \right)^{1-\lambda} \left(q_{jm}^{(\ell)} \right)^{\lambda} \right)  \right),
\end{align*}
where we have used the inequality $1-x \leq e^{-x}$. Since $n_m = (1 + o(1)) \pi_m n$, we have
\begin{align*}
\mathbb{E}\left[e^{-\lambda \left(\langle W(i) - W(j) , d(v) \rangle\right)} \right] &\leq \exp\left(-(1+o(1))t \log (n) \sum_{m \in [k]} \pi_m \left(1-\sum_{\ell \in [L]} \left(q_{im}^{(\ell)} \right)^{1-\lambda} \left(q_{jm}^{(\ell)} \right)^{\lambda} \right) \right)\\
&=\exp\left(-(1+o(1))t \log (n) \left(1 -\sum_{m \in [k]} \pi_m \sum_{\ell \in [L]} \left(q_{im}^{(\ell)} \right)^{1-\lambda} \left(q_{jm}^{(\ell)} \right)^{\lambda} \right) \right).
\end{align*}
Let $\lambda^{\star} \in [0,1]$ be such that
\[D_+(\theta_i, \theta_j) = 1 -\sum_{m \in [k]} \pi_m \sum_{\ell \in [L]} \left(q_{im}^{(\ell)} \right)^{1-\lambda^{\star}} \left(q_{jm}^{(\ell)} \right)^{\lambda^{\star}} \]
(recalling \eqref{eq:D_+-prob}). We will set $\lambda = \lambda^{\star}$, which requires showing that $\lambda^{\star} \neq 0$. To this end, let $f(\lambda) = \sum_{m \in [k]} \pi_m \sum_{\ell \in [L]} \left(q_{im}^{(\ell)} \right)^{1-\lambda} \left(q_{jm}^{(\ell)} \right)^{\lambda}$, and observe that by the inequality of arithmetic and geometric means, we have that $f(1/2) < \frac{1}{2}\left(f(0) + f(1) \right)$ and at the same time $f(0) = f(1) = 1$. It follows that $\lambda^{\star} \neq 0$.

Substituting $\lambda = \lambda^{\star}$ into \eqref{eq:Chernoff}, we obtain
\begin{align}
\mathbb{P}\left(\langle W(i) - W(j), d(v)\rangle \leq \delta \log n \right)&\leq n^{\lambda^{\star} \delta} n^{-(1+o(1)) t D_+(\theta_i, \theta_j)} \label{eq:Chernoff-bound}
\end{align}
Since $t > t_c$, it follows that $t D_+(\theta_i, \theta_j) > 1$. Therefore, we can choose $\delta, \eta > 0$ such that \eqref{eq:Chernoff-bound} is less than $n^{-1-\eta}$ for $n$ sufficiently large. While the choice of $\delta, \eta$ is specific to the fixed pair $(i,j)$, we can take the minimum over all such pairs to ensure that \eqref{eq:degree-profile} holds simultaneously for all such pairs.
\end{proof}

\begin{proof}[Proof of Theorem \ref{theorem:spectral}]
Lemma \ref{lemma:entrywise} implies that for all $\ell \in [L], i \in [k]$ we have
\[\min_{s \in \{\pm 1\}} \left \Vert s u_i^{(\ell)} - \frac{A^{(\ell)} u_i^{\star (\ell)}}{\lambda_i^{\star (\ell)}} \right \Vert_{\infty} \leq \frac{C}{\log \log(n) \sqrt{n}}\]
with probability $1-O(n^{-3})$.

Using the triangle inequality, we see that there exist $\{s_j^{(\ell)}\}_{j \in [k], \ell \in [L]} \in \{\pm 1\}^{kL}$ such that for any $i \in [k]$,
\[\left \Vert \sum_{\ell \in [L]} \sum_{j \in [k]}  c_{ij}^{(\ell)}\left(s_j^{(\ell)}u_j^{(\ell)} - \frac{A^{(\ell)} u_j^{\star (\ell)}}{\lambda_j^{\star (\ell)}} \right) \right \Vert_{\infty} \leq \frac{C'}{\log \log(n) \sqrt{n}},\]
where the weights $\{c_{ij}^{(\ell)}\}_{i,j \in [k], \ell \in [L]}$ are computed in Line \ref{step:weights} of Algorithm \ref{alg:spectral}. It follows that for some choice of $\{s_j^{(\ell)}\}_{j \in [k], \ell \in [L]} \in \{\pm 1\}^{kL}$, we have that
\begin{align}
\left\Vert \sum_{\ell=1}^L\left(U^{(\ell)}  D^{\sss (s^{(\ell)})}  c_i^{(\ell)}\right) -\sum_{\ell \in [L]} A^{(\ell)} \sum_{j \in [k]} \frac{c_{ij}^{(\ell)} u_j^{\star (\ell)}}{\lambda_j^{\star (\ell)}} \right\Vert_{\infty} &\leq \frac{C'}{\log \log (n) \sqrt{n}} \label{eq:entrywise-alg}    
\end{align}
with probability $1-O(n^{-3})$.

Next, consider the scaled vector \[x_i^{(\ell)} := \sqrt{n} \log(n) \sum_{j \in [k]} \frac{c_{ij}^{(\ell)} u_j^{\star (\ell)}}{\lambda_j^{\star (\ell)}},\] which takes a block form (i.e. $x_i^{(\ell)}$ is constant over $\{v : \sigma^{\star}(v) = j\}$, for all $j \in [k]$). 
The matrix $B^{(\ell)}$ constructed in Algorithm \ref{alg:find-weights} is equal to $\mathbb{E}[A^{(\ell)} \mid \sigma^{\star}]$ up to fluctuations in community sizes and permutation of the rows and columns. A straightforward perturbation argument (see \cite[Lemma 5.3]{Dhara2022a}) implies that $x_{i,v}^{(\ell)}  = (1+o(1)) \log \left(q^{(\ell)}_{ij}\right)$ for all $v$ with $\sigma^{\star}(v) = j$.

It follows that for all $v \in V$
\begin{align*}
\sqrt{n} \log(n) \left(\sum_{\ell \in [L]} A^{(\ell)} \sum_{j \in [k]} \frac{c_{ij}^{(\ell)} u_j^{\star (\ell)}}{\lambda_j^{\star (\ell)}}\right)_v &= (1+o(1)) \left(\sum_{\ell \in [L]} A^{(\ell)}x_i^{(\ell)}\right)_v = (1+o(1))\sum_{\ell \in [L]} \sum_{w \neq v} A^{(\ell)}_{vw} x_{iw}^{(\ell)}\\
&= (1+o(1)) \sum_{\ell \in [L]} \sum_{j \in [k]}  \sum_{w \neq v} \mathbbm{1}\{\sigma^{\star}(w) = j, \mathcal{L}(v,w) = \ell\} \log \left(q_{ij}^{(\ell)} \right)\\
&= (1+o(1)) \langle W(i), d(v) \rangle,
\end{align*}
where the second-last step is due to Step \ref{step:z} of Algorithm \ref{alg:find-weights}, and $W(i)$ is as in Lemma \ref{lemma:concentration}.

Substituting back into \eqref{eq:entrywise-alg}, we see that for all $v \in V$
\begin{align*}
\left| \sqrt{n} \log(n)\sum_{\ell=1}^L\left(U^{(\ell)}  D^{\sss (s^{(\ell)})}  c_i^{(\ell)}\right) - (1+o(1)) \langle W(i), d(v) \rangle \right| &\leq \frac{C' \log n}{\log \log n} = o(\log n). 
\end{align*}
Lemma \ref{lemma:concentration} then guarantees that \[\argmax_{j \in [k]} \sum_{\ell=1}^L\left(U^{(\ell)}  D^{\sss (s^{(\ell)})}  c_j^{(\ell)}\right) = i\]
with probability $1-o(1/n)$.

Finally, since the correct signs $\{s^{\ell}\}_{\ell  = 1}^{L} \subset \mathbb{R}^k$ are unknown, we construct all possible labelings for the different sign combinations in Step \ref{step:argmax} of Algorithm \ref{alg:spectral}, by choosing the resulting labeling with largest posterior probability. Since we have shown that the correct labeling is among these candidates, and the MAP estimator succeeds above the threshold due to Theorem \ref{theorem:IT}, it follows that $\hat{\sigma}$ exactly recovers the community partition, with high probability.

\end{proof}

\section{Conclusion}
\label{sec:conclusion}
In this paper, we proposed a simple spectral algorithm for exact recovery in the LSBM, and showed that it achieves the information-theoretic threshold for nearly all parameters (namely, whenever $Q^{(\ell)} \cdot \diag(\pi)$ has $k$ distinct, nonzero eigenvalues for all $\ell \in [L]$). The result generalizes the case $L=2$ covered by \cite{Dhara2022c}, under the stated assumption. It may be possible to relax this condition by using a non-binary encoding for each matrix $A^{(\ell)}$, using up to $L+1$ numeric values in each matrix, one for each label.

An interesting question is whether $L$ matrices are necessary for achieving exact recovery down to the information-theoretic threshold. We expect that for all but a measure-zero set of parameters, $L$ matrices are necessary. Indeed, such a result is known for the case $L = 2$ \cite[Theorem 7 (2)]{Dhara2022c}.

\section*{Acknowledgments}
J.G. and H.L. were partially supported by NSF CCF-2154100. Part of this work was completed while H.L. was at the Department of Industrial Engineering and Management Sciences, Northwestern University.

\bibliography{references}

\begin{thebibliography}{10}

\bibitem{Abbe2017}
Emmanuel Abbe.
\newblock Community detection and stochastic block models: recent developments.
\newblock {\em The Journal of Machine Learning Research}, 18(1):6446--6531,
  2017.

\bibitem{Abbe2014}
Emmanuel Abbe, Afonso~S Bandeira, Annina Bracher, and Amit Singer.
\newblock Decoding binary node labels from censored edge measurements: Phase
  transition and efficient recovery.
\newblock {\em IEEE Transactions on Network Science and Engineering},
  1(1):10--22, 2014.

\bibitem{Abbe2019}
Emmanuel Abbe, Jianqing Fan, Kaizheng Wang, and Yiqiao Zhong.
\newblock Entrywise eigenvector analysis of random matrices with low expected
  rank.
\newblock {\em Annals of Statistics}, 48(3):1452, 2020.

\bibitem{Abbe2015}
Emmanuel Abbe and Colin Sandon.
\newblock Community detection in general stochastic block models: Fundamental
  limits and efficient algorithms for recovery.
\newblock In {\em 2015 IEEE 56th Annual Symposium on Foundations of Computer
  Science}, pages 670--688. IEEE, 2015.

\bibitem{Deng2021}
Shaofeng Deng, Shuyang Ling, and Thomas Strohmer.
\newblock Strong consistency, graph {Laplacians}, and the stochastic block
  model.
\newblock {\em Journal of Machine Learning Research}, 22(117):1--44, 2021.

\bibitem{Dhara2022b}
Souvik Dhara, Julia Gaudio, Elchanan Mossel, and Colin Sandon.
\newblock Spectral algorithms optimally recover planted sub-structures.
\newblock {\em arXiv preprint arXiv:2203.11847}, 2022.

\bibitem{Dhara2022a}
Souvik Dhara, Julia Gaudio, Elchanan Mossel, and Colin Sandon.
\newblock Spectral recovery of binary censored block models.
\newblock In {\em Proceedings of the 2022 Annual ACM-SIAM Symposium on Discrete
  Algorithms (SODA)}, pages 3389--3416. SIAM, 2022.

\bibitem{Dhara2022c}
Souvik Dhara, Julia Gaudio, Elchanan Mossel, and Colin Sandon.
\newblock The power of two matrices in spectral algorithms for community
  recovery.
\newblock {\em IEEE Transactions on Information Theory}, 2023.

\bibitem{Dreveton2023}
Maximilien Dreveton, Felipe Fernandes, and Daniel Figueiredo.
\newblock Exact recovery and bregman hard clustering of node-attributed
  stochastic block model.
\newblock {\em Advances in Neural Information Processing Systems}, 36, 2023.

\bibitem{Gaudio2023}
Julia Gaudio and Nirmit Joshi.
\newblock Community detection in the hypergraph {SBM}: Exact recovery given the
  similarity matrix.
\newblock In {\em The Thirty Sixth Annual Conference on Learning Theory}, pages
  469--510. PMLR, 2023.

\bibitem{Gaudio2024}
Julia Gaudio and Nirmit Joshi.
\newblock Exact community recovery (under side information): Optimality of
  spectral algorithms.
\newblock {\em arXiv preprint arXiv:2406.13075}, 2024.

\bibitem{Hajek2016}
Bruce Hajek, Yihong Wu, and Jiaming Xu.
\newblock Achieving exact cluster recovery threshold via semidefinite
  programming: Extensions.
\newblock {\em IEEE Transactions on Information Theory}, 62(10):5918--5937,
  2016.

\bibitem{Heimlicher2012}
Simon Heimlicher, Marc Lelarge, and Laurent Massouli{\'e}.
\newblock Community detection in the labelled stochastic block model.
\newblock {\em arXiv preprint arXiv:1209.2910}, 2012.

\bibitem{Holland1983}
Paul~W Holland, Kathryn~Blackmond Laskey, and Samuel Leinhardt.
\newblock Stochastic blockmodels: First steps.
\newblock {\em Social Networks}, 5(2):109--137, 1983.

\bibitem{Lelarge2015}
Marc Lelarge, Laurent Massouli{\'e}, and Jiaming Xu.
\newblock Reconstruction in the labelled stochastic block model.
\newblock {\em IEEE Transactions on Network Science and Engineering},
  2(4):152--163, 2015.

\bibitem{McSherry2001}
Frank McSherry.
\newblock Spectral partitioning of random graphs.
\newblock In {\em Proceedings 42nd IEEE Symposium on Foundations of Computer
  Science}, pages 529--537. IEEE, 2001.

\bibitem{Yun2016}
Se-Young Yun and Alexandre Proutiere.
\newblock Optimal cluster recovery in the labeled stochastic block model.
\newblock {\em Advances in Neural Information Processing Systems}, 29, 2016.

\end{thebibliography}
\bibliographystyle{plain}

\section*{Appendix}
We now show how Theorem \ref{theorem:IT} follows from \cite[Theorem 3]{Yun2016}. Note that \cite{Yun2016} uses a more general parameterization, letting $p_{ij}^{(\ell)}$ be the probability that two vertices from communities $i,j$ have label $\ell \in \{0,1, \dots, L\}$. Our parameterization takes $p_{ij}^{(0)} = 1 - \frac{t \log n}{n}$ and $p_{ij}^{(\ell)} = \frac{t \log n}{n} \cdot q_{ij}^{(\ell)}$ for $\ell \in [L]$. We continue to use $\pi \in \mathbb{R}^k$ to denote the community label distribution. 

The result of \cite{Yun2016} is stated in terms of a different divergence $D(\pi, p)$, which is related to the CH divergence. We use $\mathcal{P}^{a \times b}$ to denote the set of $a \times b$ matrices with each row representing a probability distribution. Let $p_{ij}^{(\cdot)} = \left(p_{ij}^{(0)}, p_{ij}^{(1)}, \dots, p_{ij}^{(L)} \right)$ represent the label distribution for two vertices in communities $i,j \in [k]$. Similarly, let $p_i$ denote the $k \times (L+1)$ matrix whose $j,\ell$ entry is $p_{ij}^{(\ell)}$.
\begin{definition} 
For $i,j \in [k]$, let
\small
\begin{align*}
D_{L+}(\pi, p_i, p_j) = \inf_{y \in \mathcal{P}^{K \times (L+1)}} \max \left\{\sum_{m=1}^k \pi_m D_{KL}\left(y(k,\cdot) \Vert p_{im}^{(\cdot)} \right), \sum_{m=1}^k \pi_m D_{KL}\left(y(k,\cdot)\Vert p_{jm}^{(\cdot)} \right)  \right\},  
\end{align*}
\normalsize
where $D_{KL}(\cdot \Vert \cdot)$ is the KL divergence; i.e., $D_{KL}(P \Vert Q) = \sum_{x \in \mathcal{X}} P(x) \log \left(\frac{P(x)}{Q(x)} \right)$ for discrete distributions $P,Q$ with support $\mathcal{X}$.
Define
\[D(\pi,p) = \min_{i \neq j} D_{L+}(\pi, p_i, p_j).\]
\end{definition}

Finally, the result is stated in terms of the following assumptions.
\begin{enumerate}
    \item\label{assumption1} There exists $\eta >0$ such that for all $i,j,m \in [k]$ and $\ell \in \{0,1, \dots, L\}$, 
    \[\nicefrac{p_{ij}^{(\ell)}}{p_{im}^{(\ell)}} \leq \eta\]
    \item\label{assumption2} There exists $\epsilon > 0$ such that for all $i,j \in [k]$ with $i \neq j$,
    \[\sum_{m=1}^k\sum_{\ell = 1}^L \left(p_{im}^{(\ell)} -p_{jm}^{(\ell)}\right)^2 \geq \epsilon \overline{p}^2,\]
    where $\overline{p} = \max_{i,j \in [k], \ell \in [L]} p_{ij}^{(\ell)}$.
    \item\label{assumption3} There exists $\kappa > 0$ such that for all $i,j \in [k], \ell \in [L]$, it holds that
    \[n p_{ij}^{(\ell)} \geq (n\overline{p})^{\kappa}.\]
\end{enumerate}
\begin{theorem}[Theorem 3, \cite{Yun2016}]\label{theorem:IT-prior-work}
Consider the LSBM with parameters $\pi \in \mathbb{R}^k$ and $\{p_{ij}^{(\ell)}: i,j \in [k], \ell \in \{0, 1, \dots, L\}\}$. Suppose Assumptions \ref{assumption1} and \ref{assumption2} hold. If the parameters $\pi, p$ are such that $\liminf_{n \to \infty} \frac{n D(\pi,p)}{\log n} < 1$, then no algorithm achieves exact recovery. Moreover, if $\frac{n D(\pi,p)}{\log n} < 1-\delta$ for some $\delta > 0$, then any algorithm fails to identify the community partition with probability at least $1 - n^{-\frac{\delta}{4}}$.

Conversely, if $\liminf_{n \to \infty} \frac{n D(\pi,p)}{\log n} \geq 1$ and Assumptions \ref{assumption1}-\ref{assumption3} hold, then the Spectral Partition Algorithm \cite{Yun2016} achieves exact recovery. 
\end{theorem}

Finally, we prove Theorem \ref{theorem:IT} in light of Theorem \ref{theorem:IT-prior-work}.
\begin{proof}[Proof of Theorem \ref{theorem:IT}]
We first show that Assumptions \ref{assumption1}-\ref{assumption3} hold trivially under our parameterization. Observe that Assumption \ref{assumption1} holds with \[\eta = \max_{i,j,m \in [k], \ell \in [L]} \frac{q_{ij}^{(\ell)}}{q_{im}^{(\ell)}}.\]  
Similarly, Assumption \ref{assumption2} holds with
\[\epsilon = \left(\max_{i,j \in [k], \ell \in [L]} q_{ij}^{(\ell)} \right)^{-2} \cdot \min_{i,j \in [k], i \neq j} \sum_{m=1}^k\sum_{\ell = 1}^L \left(q_{im}^{(\ell)} -q_{jm}^{(\ell)}\right)^2. \]
Finally, Assumption \ref{assumption3} holds with any $\kappa \in (0,1)$ for $n$ sufficiently large.

Next, we invoke \cite[Claim 4]{Yun2016} to conclude that for all $i,j \in [k]$,
\begin{align*}
D_{L+}(\pi, p_i, p_j) &\sim \frac{t \log n}{n} \sup_{\lambda \in [0,1]} \sum_{\ell =1}^L \sum_{m=1}^k \pi_m \left((1-\lambda) q_{im}^{(\ell)} + \lambda q_{jm}^{(\ell)} - (q_{im}^{(\ell)})^{1-\lambda} (q_{jm}^{(\ell)})^{\lambda}\right)\\
&= \frac{t\log n}{n} D_{+}(\theta_i,\theta_j).
\end{align*}
Here $a_n \sim b_n$ means $\lim_{n\to \infty} \frac{a_n}{b_n} = 1$.
Observe that if 
\[t < \frac{1}{\min_{i \neq j} D_+(\theta_i, \theta_j)}~~~ (= t_c),\]
then for some $i\neq j$ and $\delta > 0$ we have $t D_+(\theta_i,\theta_j) < 1-\delta$. It follows that for all $n$ sufficiently large, $ \frac{n D(\pi,p)}{\log n}< 1-\frac{\delta}{2}$. Theorem \ref{theorem:IT-prior-work} then implies that any estimator fails with probability $1-o(1)$.

On the other hand, if
\[t > \frac{1}{\min_{i \neq j} D_+(\theta_i, \theta_j)},\]
then $t D_+(\theta_i,\theta_j) > 1$ for all $i\neq j$. It follows that $\liminf_{n \to \infty} \frac{n D(\pi,p)}{\log n}> 1$, and so the second part of Theorem \ref{theorem:IT-prior-work} implies that the MAP estimator succeeds in exact recovery.

\end{proof}

\end{document}